\documentclass[11pt]{article}
\usepackage{amscd, amsmath, amssymb, amsthm}
\usepackage[all,cmtip]{xy}
\usepackage[pagebackref]{hyperref}

\title{Klt varieties of general type with small volume}
\author{Burt Totaro and Chengxi Wang}
\date{  }

\def\N{\text{\bf N}}
\def\Z{\text{\bf Z}}
\def\Q{\text{\bf Q}}

\def\C{\text{\bf C}}
\def\P{\text{\bf P}}

\DeclareMathOperator{\vol}{vol}

\def\d{\widetilde{d}}

\begin{document}
\maketitle
\newtheorem{theorem}{Theorem}[section]
\newtheorem{corollary}[theorem]{Corollary}
\newtheorem{lemma}[theorem]{Lemma}

\theoremstyle{definition}
\newtheorem{definition}[theorem]{Definition}
\newtheorem{example}[theorem]{Example}

\theoremstyle{remark}
\newtheorem{remark}[theorem]{Remark}

For projective varieties of general type,
the {\it volume }measures the asymptotic growth
of the plurigenera: $\vol(X)=\limsup_{m\to\infty} h^0(X,mK_X)/(m^n/n!)$.
This is equal to the intersection number $K_X^n$
if the canonical class $K_X$ is ample.
A central fact
about the classification
of algebraic varieties is the theorem of Hacon-M\textsuperscript{c}Kernan-Xu,
which says in particular:
for mildly singular (klt) complex projective varieties $X$
with ample canonical class,
there is a constant
$s_n$ depending only on the dimension $n$ of $X$
such that the pluricanonical linear system
$|mK_X|$ gives a birational embedding of $X$ into projective
space for all $m\geq s_n$ \cite[Theorem 1.3]{HMXacc}.
It follows that there is a positive lower bound
$v_n$ for the volume of all klt $n$-folds with ample canonical class:
namely, $1/(s_n)^n$ is a lower bound. 
It is a fundamental
problem to find the optimal values of these constants.

We focus here on constructing klt varieties of general type
with small volume. (We also construct klt Fano varieties
with similar exotic behavior.)
It is also interesting to look for small volumes
in the narrower setting of varieties with canonical singularities
and ample canonical class,
since these arise as canonical models of smooth projective varieties
of general type. In that direction,
Ballico, Pignatelli, and Tasin constructed
smooth projective $n$-folds of general type
with volume about $1/n^n$ \cite[Theorems 1 and 2]{BPT}.
After several advances, Esser and the authors
constructed smooth projective $n$-folds
of general type with volume about
$1/2^{2^{n/2}}$ \cite[Theorem 1.1]{ETW}.
That paper also gives comparably
extreme examples of Fano and Calabi-Yau varieties.
Returning to the klt setting, our examples here have volume
roughly $1/2^{2^n}$.
These examples should actually be close to optimal,
by the following discussion.

In the more general context of klt pairs,
Koll\'ar proposed what may be the klt pair $(Y,\Delta)$
of general type with standard coefficients that has minimum volume
\cite[Introduction]{HMXbir}.
(Here ``general type'' means that $K_Y+\Delta$
is big, and ``standard coefficients'' means
that all coefficients of the $\Q$-divisor $\Delta$
are of the form $1-1/m$ for $m\in\Z^{+}$.)
There is some positive lower bound for such volumes,
and the minimum is attained,
by Hacon-M\textsuperscript{c}Kernan-Xu's theorem
that these volumes satisfy DCC \cite[Theorem 1.3]{HMXacc}.
The example is
$$(Y,\Delta)=\bigg( \P^n,\frac{1}{2}H_0+\frac{2}{3}H_1
+\frac{6}{7}H_2+\cdots+\frac{c_{n+1}-1}{c_{n+1}}H_{n+1}\bigg),$$
where $H_0,H_1,\ldots,H_{n+1}$ are $n+2$ general hyperplanes
and $c_0,c_1,c_2,\ldots\,$ is Sylvester's sequence,
$$c_0=2\text{ and }c_{m+1}=c_m(c_m-1)+1.$$
In this case, the volume of $K_Y+\Delta$ is $1/(c_{n+2}-1)^n$,
which is really small, less than $1/2^{2^n}$.
The optimality of Koll\'ar's example is known only in dimension 1,
where it is the ``Hurwitz orbifold'' of volume $1/42$
\cite[section 10]{AL}.

How small can the volume be for a klt variety with ample
canonical class, as opposed to a klt pair?
In dimension 2, Alexeev and Liu gave an example 
with volume $1/48983$ \cite[Theorem 1.4]{AL}.
In high dimensions, we give examples as follows (Theorems \ref{kltr=3}
and \ref{kltgeneral}). Following a long tradition in algebraic
geometry \cite{Iano-Fletcher, JK, BPT, BK},
our examples are weighted projective hypersurfaces.
These exhibit a huge range of behavior, and finding good examples
is not easy.

\begin{theorem}
\label{kltintro}
For every integer $n\geq 2$,
there is a complex klt $n$-fold $X$ with ample canonical class such that
$\vol(K_X)<1/2^{2^n}$. More precisely,
$\log(\vol(K_X))$ is asymptotic to $\log(\vol(K_Y+\Delta))$
as the dimension goes to infinity,
where $(Y,\Delta)$ is Koll\'ar's klt pair above.
\end{theorem}

Since Koll\'ar's example is conjecturally optimal in the broader
setting of klt pairs with standard coefficients, Theorem
\ref{kltintro} means that our klt varieties
with ample canonical class should be close
to optimal in high dimensions.
The details of the construction
are intricate, combining Sylvester's sequence
with several sequences of polynomials defined by recurrence
relations.

Finally, we construct a klt Fano variety $X$ in every dimension $n$
such that the linear system $|-mK_X|$ is empty for all $1\leq m< b$,
with $b$ doubly exponential in $n$ (Theorem \ref{kltfano}). More precisely,
$b$ is roughly $2^{2^n}$. (In the narrower setting of terminal
Fano varieties, Esser and the authors gave examples
with $b$ roughly $2^{2^{n/2}}$ \cite[Theorem 3.9]{ETW}.)
Birkar's theorem on the boundedness of complements implies that there
is an upper bound on the number of vanishing spaces of sections
$H^0(X,-mK_X)$, for all klt Fano varieties
of a given dimension \cite[Theorem 1.1]{Birkarcomp}.
Our examples show that the bound must grow extremely fast
as the dimension increases.

This work was supported by NSF grant DMS-1701237. We thank
the Simons Foundation for making Magma available
to U.S. universities.
Thanks to Jungkai Chen, Louis Esser,
and Miles Reid for their suggestions.

\section{Background on weighted projective spaces}

Some introductions to the singularities
of the minimal model program, such as terminal, canonical, or
Kawamata log terminal (klt),
are \cite{ReidYPG, KM}. We work over $\C$, although much of
the following would work in any characteristic.

For positive integers $a_0,\ldots,a_n$,
the weighted projective space $Y=P(a_0,\ldots,a_n)$
is said to be {\it well-formed }if $\gcd(a_0,\ldots,\widehat{a_j},
\ldots,a_n)=1$ for each $j$. We always assume this.
(In other words, the analogous
quotient stack $[(A^{n+1}-0)/G_m]$,
where the multiplicative group $G_m$ acts by
$t(x_0,\ldots,x_n)=(t^{a_0}x_0,\ldots,t^{a_n}x_n)$,
has trivial stabilizer group
in codimension 1.) For well-formed $Y$, the canonical divisor of $Y$
is given by $K_Y=O(-a_0-\cdots-a_n)$ \cite[Theorem 3.3.4]{Dolgachev}.
Here $O(m)$ is the reflexive sheaf
associated to a Weil divisor for any integer $m$;
it is a line bundle if and only if $m$ is a multiple of every
weight $a_i$. The intersection number $\int_Y c_1(O(1))^n$
is equal to $1/a_0\cdots a_n$. (To check this,
think of the intersection number as $\vol(O(1))$, and use that
the coordinate ring of $O(1)$ is the graded polynomial ring with generators
in degrees $a_0,\ldots,a_n$.)

Since weighted projective spaces have quotient singularities,
they are klt.
A closed subvariety $X$ of a weighted projective space $P(a_0,\ldots,a_n)$
is called {\it quasi-smooth }if its affine cone
in $A^{n+1}$ is smooth outside
the origin. It follows that $X$ has only cyclic quotient singularities.
A subvariety $X$ in $Y=P(a_0,\ldots,a_n)$
is said to be {\it well-formed }if $Y$ is well-formed and the codimension
of $X\cap Y^{\text{sing}}$ in $X$ is at least 2. Notably,
the adjunction formula holds for a well-formed
quasi-smooth hypersurface $X$ of degree $d$ in $Y$, meaning
that $K_X=O_X(d-\sum a_i)$
\cite[section 6.14]{Iano-Fletcher}. A general hypersurface
of degree $d$ is well-formed if and only if $\gcd(a_0,\ldots,
\widehat{a_i},\ldots,\widehat{a_j},\ldots,a_n)|d$ for all $i<j$
\cite[section 6.10]{Iano-Fletcher};
that holds for all examples in this paper. Indeed, assuming 
that $d$ is not equal to any $a_i$ (as will be true
in our examples), a quasi-smooth
hypersurface of dimension at least 3 is always
well-formed \cite[Theorem 6.17]{Iano-Fletcher}.

Iano-Fletcher proved the following criterion for quasi-smoothness,
using that we are in characteristic zero
\cite[Theorem 8.1]{Iano-Fletcher}. Here $\N$ denotes the natural numbers,
$\{0,1,\ldots\}$.

\begin{lemma}
\label{quasi}
A general hypersurface of degree $d$ in $P(a_0,\ldots,a_n)$
is quasi-smooth if and only if 

either (1) $a_i=d$ for some i,

or (2) for every nonempty subset $I$ of $\{0,\ldots,n\}$,
either (a) $d$ is an $\N$-linear combination of the numbers
$a_i$ with $i\in I$,
or (b) there are at least $|I|$ numbers $j\not\in I$ such that
$d-a_j$ is an $\N$-linear combination of the numbers $a_i$
with $i\in I$.
\end{lemma}

\section{Klt varieties with ample canonical class}
\label{kltsectr=3}

As in the introduction, let $c_0,c_1,c_2,\ldots$ be Sylvester's sequence
\cite{OEIS},
$$c_0=2\text{ and }c_{n+1}=c_n(c_n-1)+1.$$
The first few terms are $c_0=2$, $c_1=3$, $c_2=7$, $c_3=43$, $c_4=1807$.
We give the following examples of klt varieties with ample canonical
class. We will generalize the construction as Theorem \ref{kltgeneral},
giving better but more complicated examples.

\begin{theorem}
\label{kltr=3}
Let $n$ be an integer at least $2$,
and define integers $a_0,\ldots,a_{n+1}$ as follows.
Let $y=c_{n-1}-1$ and
\begin{align*}
a_2&=y^3+y+1\\
a_1&=y(y+1)(1+a_2)-a_2\\
a_0&=y(1+a_2+a_1)-a_1.
\end{align*}
Let
$x=1+a_0+a_1+a_2$,
$d=yx=c_0\cdots c_{n-2}x=y^7+y^6+y^5+4y^4+2y^3+2y^2+2y$,
and $a_{i+3}=c_0\cdots \widehat{c_i}\cdots c_{n-2}x$ for $0\leq i\leq n-2$.
Let $X$ be a general hypersurface of degree $d$
in the complex weighted projective space $P(a_0,\ldots,a_{n+1})$.
Then $X$ is a klt projective variety of dimension $n$ with ample
canonical class, and
$$\vol(K_X)=\frac{1}{y^{n-3}x^{n-2}a_0a_1a_2}.$$
It follows that $\vol(K_X)<\frac{1}{(c_{n-1}-1)^{7n-1}}$
and hence $\vol(K_X)<\frac{1}{2^{2^n}}$.
\end{theorem}

This example is not optimal, but it should be fairly close
to optimal, given the fast-growing functions involved.
Indeed, Koll\'ar's conjecturally optimal klt pair $(Y,\Delta)$
from the introduction has $\vol(K_Y+\Delta)=1/(c_{n+2}-1)^n
\doteq 1/(c_{n-1}-1)^{8n}$, while the klt variety $X$ in Theorem \ref{kltr=3}
has $\vol(K_X)\doteq 1/(c_{n-1}-1)^{7n-1}$, thus about
the $7/8$th power of the volume of Koll\'ar's klt pair.
See Theorem \ref{kltgeneral} for a generalization, producing
better examples.

Some cases of Theorem \ref{kltr=3} in low dimensions,
klt varieties with ample canonical class, are:

$X_{316}\subset P^3(158,85,61,11)$ of dimension 2, with volume
$2/57035\doteq 3.5 \times 10^{-5}$.

$X_{340068}\subset P^4(170034,113356,47269,9185,223)$ of dimension 3,
with volume
$$1/5487505331993410\doteq 1.8\times 10^{-16}.$$

The klt 4-fold with ample canonical class
given by Theorem \ref{kltr=3} has volume about $1.4\times
10^{-44}$. For comparison, the smallest known volume
for a klt 4-fold with ample canonical class is about $1.4\times 10^{-47}$
\cite[ID 538926]{GRD}.

\begin{proof}
Sylvester's sequence satisfies $c_{m}=c_0\cdots c_{m-1}+1$.
It follows that any two terms in the sequence are relatively prime.
Another notable feature is that
$$\frac{1}{2}+\frac{1}{3}+\frac{1}{7}+\cdots+\frac{1}{c_{m}}
=1-\frac{1}{c_{m+1}-1},$$
which converges very quickly to 1 as $m$ increases.

We first show that the weighted projective space
$Y=P(a_0,\ldots,a_{n+1})$ is well-formed.
That is, we have to show
that $\gcd(a_0,\ldots,\widehat{a_m},\ldots,a_{n+1})=1$
for each $0\leq m\leq n+1$.
It suffices to show that $a_0,a_1,a_2$
are pairwise relatively prime.

Indeed, 1 is a $\Z[y]$-linear combination of any two
of $a_0,a_1,a_2$. For clarity,
however, let us check by hand that $a_0,a_1,a_2$
are pairwise relatively prime.
To show that $\gcd(a_1,a_2)=1$,
let $p$ be a prime number dividing $a_1$ and $a_2$. By the formula
for $a_1$, $p$ divides $y(y+1)$. 
But as a polynomial in $y$, $a_2(y)=y^3+y+1$
satisfies $a_2(0)=1$
and $a_2(-1)=-1$; so $a_2\equiv 1\pmod{y}$ and $a_2\equiv -1\pmod{y+1}$.
This contradicts that $p$ divides
$a_2$. So $\gcd(a_1,a_2)=1$.

Next, let $p$ be a prime number that divides $a_0$
and $a_1$. By the formula for $a_0$, either $p$ divides
$1+a_2$ or $p$ divides $y$. In both cases, the formula
for $a_1$ gives that $p$ divides $a_2$, contradicting
that $\gcd(a_1,a_2)=1$. So $\gcd(a_0,a_1)=1$.

Finally, let $p$ be a prime number that divides $a_0$
and $a_2$. Then the formulas for $a_0$ and $a_1$
imply that $a_1\equiv y(a_1+1)\pmod{p}$, that is, $(y-1)a_1+y
\equiv 0\pmod{p}$, and $a_1\equiv y(y+1)\pmod{p}$. Combining these
shows that $(y-1)y(y+1)+y=y^3\equiv 0\pmod{p}$.
So $p$ divides $y$.
But then $a_1\equiv y(y+1)\equiv 0\pmod{p}$, contradicting
that $\gcd(a_1,a_2)=1$. It follows that the weighted
projective space $Y$ is well-formed.

Next, let us show that the general hypersurface $X$ of degree $d$
in $Y$ is quasi-smooth. We use the following sufficient condition
in terms of a cycle of congruences.

\begin{lemma}
\label{cycle}
For positive integers $d$ and $a_0,\ldots,a_{n+1}$,
a general hypersurface of degree $d$ in $P(a_0,\ldots,a_{n+1})$
is quasi-smooth if $d\geq a_i$ for every $i$
and there is a positive integer $r$ such that:

(1) $a_i|d$ if $i\geq r$ (that is, all but the first $r$ weights
divide $d$),

and (2) $d-a_{r-1}\equiv 0
\pmod{a_{r-2}}$, \ldots, $d-a_1
\equiv 0\pmod{a_0}$, and $d-a_0\equiv 0\pmod{a_{r-1}}$.
\end{lemma}

\begin{proof} Use Lemma \ref{quasi}.
We have something to prove for each nonempty subset $I$
of $\{0,\ldots,n+1\}$. If $I$ contains a number $i\geq r$,
then $a_i$ divides $d$ and we are done. Otherwise, $I$ is contained
in the set $S=\{0,\ldots,r-1\}$. Consider $S$ as the vertices
of a directed graph,
with arrows from $r-1$ to $r-2$ to \ldots to 0 to $r-1$; then $S$
is a directed cycle of length $r$. If $I$ contains two vertices $j,i$
with an edge from $j$ to $i$, then the congruence $d-a_j\equiv 0\pmod{a_i}$
implies that $d$ is an $\N$-linear combination of $a_i$ and $a_j$,
using that $d\geq a_j$, and we are done.

Otherwise, $I$ contains no edge
in the graph. Let $J$ be the set of vertices that point to some element
of $I$. We have $|J|=|I|$, and $J$ is disjoint from $I$ because
$I$ contains no edge. For each element $j\in J$ pointing to a vertex
$i\in I$, the congruence $d-a_j\equiv 0\pmod{a_i}$ implies that
$d-a_j$ is an $\N$-linear combination of the numbers $a_m$
with $m\in I$. This checks the condition of Lemma \ref{quasi}
for quasi-smoothness.
\end{proof}

Returning to the proof of Theorem \ref{kltr=3},
let us use Lemma \ref{cycle} to prove that the general hypersurface $X$
of degree $d$ in $Y$ is quasi-smooth.
We know that $a_i$ divides $d$ for each $i\geq 3$; also,
$d$ is greater than every $a_i$. Given that, it suffices
to prove the cycle of 3 congruences: $d-a_2\equiv 0\pmod{a_1}$,
$d-a_1\equiv 0\pmod{a_0}$, and $d-a_0\equiv 0\pmod{a_2}$.
Using that $x=1+a_0+a_1+a_2$ and $d=yx$,
we compute that
\begin{align*}
d-a_2&=(y^2+1)a_1\\
d-a_1&=(y+1)a_0\\
d-a_0&=(y^4+3y-1)a_2,
\end{align*}
proving the desired
congruences. That completes
the proof that $X$ is quasi-smooth.
In particular, $X$ has only cyclic quotient singularities,
and so $X$ is klt.

Since $Y$ is well-formed and $X$ is quasi-smooth,
$K_X=O_X(d-\sum a_i)$. Here $d-\sum_{i=3}^{n+1}a_i
=c_0\cdots c_{n-2}(1-\sum_{i=0}^{n-2}1/c_i)x=x=1+a_0+a_1
+a_2$, and so $K_X=O_X(1)$. As a result,
\begin{align*}
\vol(K_X)&=\frac{d}{a_0\cdots a_{n+1}}\\
&=\frac{(c_0\cdots c_{n-2})x}
{(c_0\cdots c_{n-2})^{n-2}x^{n-1}
a_0a_1a_2}\\
&=\frac{1}{y^{n-3}x^{n-2}a_0
a_1a_2}.
\end{align*}

In terms of $y=c_{n-1}-1$, we have $a_2=y^3+y+1>y^3$,
$a_1=y^5+y^4+3y^2+y-1>y^5$,
$a_0=y^6+3y^3-y^2+1>y^6$,
and $x=y^6+y^5+y^4+4y^3+2y^2+2y+2>y^6$.
Therefore, $\vol(K_X)<1/y^{7n-1}=1/(c_{n-1}-1)^{7n-1}$.

There is a constant $c\doteq 1.264$ such that $c_i$ is the closest
integer to $c^{2^{i+1}}$ for all $i\geq 0$ \cite[equations
2.87 and 2.89]{GK}.
This implies the crude statement that $\vol(K_X)< 1/2^{2^n}$
for all $n\geq 2$.
\end{proof}

\section{Some polynomial sequences defined by recurrence
relations}
\label{poly}

Here we define five sequences of polynomials in $\Z[y]$ by recurrence
relations, $f_i$, $e_i$, $b_i$, $z_i$ and $d_i$.
These will be used for defining our examples of klt varieties
with ample canonical class in Theorem \ref{kltgeneral}, generalizing
Theorem \ref{kltr=3}.
It would be interesting to know
if these polynomials (such as $f_i$, below) have been encountered before.

\begin{definition}
\label{fpoly}
For each $i\geq 0$, define a polynomial $f_i$ in $\Z[y]$ by:
$f_0=y+1$, $f_1=y^2+1$,
and $f_i = f_{i-1}f_{i-2} + (f_{i-1}-1)(f_{i-1}-2)$ for $i\geq 2$.
\end{definition}

For example, $f_0=y+1$, $f_1=y^2+1$, and $f_2=y^4+y^3+y+1$. Clearly
the polynomial $f_i$ has degree $2^i$ for each $i\geq 0$.
The following description of $f_i$ may seem more natural.

\begin{lemma}
\label{flem}
For all $i\geq 0$,
$$f_i=1+y(f_0\cdots f_{i-1}-f_0\cdots f_{i-2}+\cdots+(-1)^i).$$
\end{lemma}

\begin{proof}
Temporarily define a sequence of polynomials $h_i$ in $\Z[y]$ by 
$$h_i=1+y(h_0\cdots h_{i-1}-h_0\cdots h_{i-2}+\cdots+(-1)^i).$$
We want to show that $h_i=f_i$ for all $i\geq 0$.
We have $h_0=y+1=f_0$ and $h_1=y^2+1=f_1$. It remains
to show that $h_i$ satisfies the recurrence relation that defines
$f_i$ for $i\geq 2$. We clearly have $h_i+h_{i-1}-2=yh_0\cdots h_{i-1}$.
Likewise, $h_{i-1}+h_{i-2}-2=yh_0\cdots h_{i-2}$. Therefore,
$h_i+h_{i-1}-2=h_{i-1}(h_{i-1}+h_{i-2}-2)$, which is equivalent
to the desired relation $h_i=h_{i-1}h_{i-2}+(h_{i-1}-1)(h_{i-1}-2)$.
So $h_i=f_i$ for all $i\geq 0$.
\end{proof}

The next polynomial sequence we will need is:

\begin{definition}
\label{epoly}
For each $i\geq 0$, define a polynomial $e_i$ in $\Z[y]$ by
$$e_i=yf_0\cdots f_{i-1}.$$
\end{definition}

For example, $e_0=y$, $e_1=y(y+1)=y^2+y$, and $e_2=y(y+1)(y^2+1)
=y^4+y^3+y^2+y$. By Lemma \ref{flem}, we have
$$e_i=f_i+f_{i-1}-2$$
for all $i\geq 1$, which can be viewed as an alternative
definition of $e_i$. We can also say that $e_i=f_{i-1}e_{i-1}$
for all $i\geq 1$. The polynomial $e_i$ has degree $2^i$
for each $i\geq 0$.

\begin{definition}
\label{bpoly}
For each $i\geq 0$, define a polynomial $b_i$ in $\Z[y]$ by
$b_0=1$ and
$$b_i=(-1)^i+f_{i-1}b_{i-1}$$
for $i\geq 1$.
\end{definition}

It follows by induction that
$$b_i=f_0\cdots f_{i-1}-f_1\cdots f_{i-1}+\cdots+(-1)^i$$
for all $i\geq 0$.
For example, $b_0=1$, $b_1=y$, and $b_2=y^3+y+1$ (which was the smallest
weight of the weighted projective space in Theorem \ref{kltr=3}).
The polynomial $b_i$
has degree $2^i-1$ for each $i\geq 0$.

\begin{definition}
\label{zpoly}
For each $i\geq 0$,
define a polynomial $z_i$ in $\Z[y]$ by
$z_0=y-1$, $z_1=y^2-y+1$, and 
$$z_i=e_{i-1}z_{i-1}+z_{i-2}$$
for all $i\geq 2$.
\end{definition}

For example, $z_2=y^4+2y-1$. The polynomial $z_i$
has degree $2^i$ for each $i\geq 0$. The following identity,
needed for Theorem \ref{kltgeneral}, relates the polynomial $z_i$
to $b_i$ and $f_i$, which may be considered simpler.

\begin{lemma}
\label{zrelation}
For every $i\geq 0$,
$$f_0\cdots f_{i-1}z_i=(-1)^{i+1}+b_i(f_i-1).$$
\end{lemma}

\begin{proof}
The lemma holds for $i=0$
(since $y-1=-1+1(y)$) and for $i=1$ (since $(y+1)(y^2-y+1)=1+y(y^2)$).
Now let $i\geq 2$ and assume the lemma
for smaller values of $i$.
Then the definition of $z_i$ gives that:
\begin{align*}
f_0\cdots f_{i-1}z_i{}= &\; f_0\cdots f_{i-1}(e_{r-1}z_{r-1}+z_{i-2})\\
{}= &\; (f_0\cdots f_{i-2}z_{i-1})(f_{i-1}e_{i-1})+(f_0\cdots f_{i-3}z_{i-2})
(f_{i-2}f_{i-1})\\
\begin{split}
{}= &\; \big[ (-1)^i+b_{i-1}(f_{i-1}-1)\big] e_i\\
& +\big[ (-1)^{i-1}+b_{i-2}(f_{i-2}-1)\big]
f_{i-2}f_{i-1},
\end{split}
\end{align*}
using that the lemma holds for smaller values
of $i$.
So the lemma holds for $i$ if 0 is equal to
\begin{align*}
& (-1)^{i+1}+b_i(f_i-1)-\big[ (-1)^i+b_{i-1}(f_{i-1}-1)\big] e_i\\
& -\big[ (-1)^{i-1}+b_{i-2}(f_{i-2}-1)\big] f_{i-2}f_{i-1}.
\end{align*}
By definition of $b_i$, we have $b_{i-1}=(-1)^{i-1}+b_{i-2}f_{i-2}$
and likewise $b_i=(-1)^i+b_{i-1}f_{i-1}$. So
we need to show that 0 is equal to
\begin{multline*}
\begin{aligned}
& \; \big[ (-1)^i+b_{i-1}f_{i-1}\big](-e_i+e_{i-1}-f_{i-2}+1)\\
+ & \;\big[ (-1)^{i-1}+b_{i-2}f_{i-2}\big](1-f_{i-2})f_{i-1}\\
+ & \; b_i(f_i-1)+b_{i-1}(f_{i-2}-1)f_{i-1}
+(-1)^i(f_{i-1}+f_{i-2}-e_{i-1}-2)
\end{aligned}\\
= b_i(-e_i+e_{i-1}-f_{i-2}+f_i)
+(-1)^i(f_{i-1}+f_{i-2}-e_{i-1}-2).
\end{multline*}
This is zero by the identities
$e_i=f_i+f_{i-1}-2$ and
$e_{i-1}=f_{i-1}+f_{i-2}-2$.
Lemma \ref{zrelation} is proved.
\end{proof}

The last sequence of polynomials we define is:

\begin{definition}
\label{dpoly}
For each $i\geq 0$, define a polynomial $d_i$ in $\Z[y]$ by
$$d_i=e_i+b_i(f_i-1).$$
\end{definition}

For example, $d_0 = 2y$,
$d_1 = y^3+y^2+y$, and $d_2 = y^7+y^6+y^5+4y^4+2y^3+2y^2+2y$
(which was the degree of the hypersurface in Theorem \ref{kltr=3}).
The polynomial $d_i$ has degree $2^{i+1}-1$ for each $i\geq 0$.
By Lemma \ref{zrelation}, another formula for $d_i$
is that $d_i = (-1)^i + f_0\cdots f_{i-1}(z_i+y)$.

\section{Better klt varieties with ample canonical class}

We now construct klt varieties $X$ with ample canonical class
and with smaller volume than in Theorem \ref{kltr=3}.
These should be close
to optimal in high dimensions. Indeed, we give examples
with $\log(\vol(K_X))$ asymptotic to $\log(\vol(K_Y+\Delta))$
as the dimension goes to infinity,
where $(Y,\Delta)$ is Koll\'ar's conjecturally optimal
klt pair from the introduction.

For any odd number $r\geq 3$
and any dimension $n\geq r-1$,
we give an example with weights chosen to satisfy
a cycle of $r$ congruences. For $r=3$, this is the example
in Theorem \ref{kltr=3}. For each odd $r\geq 3$,
our klt variety $X$ compares to
Koll\'ar's conjecturally optimal klt pair by 
$$\frac{\log(\vol(K_X))}{\log(\vol(K_Y+\Delta))}\to \frac{2^r-1}{2^r}$$
as $n$ goes to infinity. Thus, by increasing $r$ as $n$ increases,
we can make this ratio converge to 1.

The example given by Theorem \ref{kltgeneral} in dimension 4,
with $r=5$,
is a general hypersurface of degree $147565206676$ in
$$P^5(73782603338, 39714616165, 28421358181, 5458415771, 187980859, 232361).$$
Here $X$ has volume $\doteq 7.4\times 10^{-45}$. This is better
than the klt 4-fold given by Theorem \ref{kltr=3}, although
the smallest known volume
for a klt 4-fold with ample canonical class is about $1.4\times 10^{-47}$
\cite[ID 538926]{GRD}.

Let $c_0,c_1,\ldots$ be Sylvester's sequence;
see section \ref{kltsectr=3} for the properties of that sequence.
We also use the five sequences of polynomials $f_i$, $e_i$, $b_i$,
$z_i$ and $d_i$ in $\Z[y]$ from section \ref{poly}.

\begin{theorem}\label{kltgeneral}
Let $r$ be an odd integer at least 3 and let $n$ be an integer at least $r-1$.
Define integers $a_0,\ldots,a_{n+1}$ as follows.
Let $y=c_{n-r+2}-1$ and
\begin{align*}
a_{r-1}&=b_{r-1}\\
a_0&=d_{r-1}-(z_{r-1}+y)a_{r-1}\\
a_1&=d_{r-1}-f_0a_0\\
&\cdots\\
a_{r-2}&=d_{r-1}-f_{r-3}a_{r-3}\\
\end{align*}
These are positive integers. Let
$x=1+a_0+\cdots+a_{r-1}$; then
$d_{r-1}=yx=c_0\cdots c_{n-r+1}x$.
Let $a_{r+i}=c_0\cdots \widehat{c_i}\cdots c_{n-r+1}x$
for $0\leq i\leq n-r+1$. Let $X$ be a general hypersurface of degree $d_{r-1}$ in the complex weighted projective space $P(a_0,\ldots,a_{n+1})$.
Then $X$ is a klt projective variety of dimension $n$ with ample
canonical class, and
$$\vol(K_X)=\frac{1}{y^{n-r}x^{n-r+1}a_0\cdots a_{r-1}}.$$
It follows that $\vol(K_X)<\frac{1}{(c_{n-r+2}-1)^{(2^r-1)n-1}}$
and hence $\vol(K_X)<\frac{1}{2^{2^n}}$.
\end{theorem}

\begin{remark}
As mentioned in the introduction, Hacon-M\textsuperscript{c}Kernan-Xu
showed that for each positive integer
$n$, there is a constant $s_n$ such that for every klt projective
variety $X$ of dimension $n$
with $K_X$ ample, the linear system $|s_nK_X|$ gives a birational
embedding of $X$ into projective space \cite[Theorem 1.3]{HMXacc}.
Although we have emphasized the role of volume,
Theorem \ref{kltgeneral} shows that $s_n$ must also grow at least doubly
exponentially with $n$. Indeed, the variety $X$ has
$$H^0(X,mK_X)=0$$
for all $1\leq m<b_{r-1}$, where
$b_{r-1}\geq (c_{n-r+2}-1)^{2^{r-1}-1}$.
Taking $r=n+1$ if $n$ is even and $r=n$ if $n$ is odd, we deduce that
the bottom weight $b_{r-1}$ is at least $2^{2^n-1}$ if $n$ is even
and at least $6^{2^{n-1}-1}$ if $n$ is odd.
\end{remark}

\begin{proof}
For any positive integer $r$ and $n\geq r-1$,
define $a_0,\ldots,a_{r-1}$ as in the theorem. We start by proving
various identities that we need, leading up to the proof
that $d_{r-1}=yx$. We only introduce the assumption that $r$ is odd
and at least 3
when we prove that $Y$ is well-formed.

A first step is to show that $a_{r-1}=d_{r-1}-f_{r-2}a_{r-2}$
if $r\geq 2$.
By section \ref{poly},
we have $d_{r-1}=(-1)^{r-1}+f_0\cdots f_{r-2}(z_{r-1}+y)$.
Multiplying by $(-1)^{r-1}a_{r-1}$ gives that
\begin{align*}
a_{r-1}={}& (-1)^{r-1}d_{r-1}a_{r-1}
+(-1)^rf_0\cdots f_{r-2}(z_{r-1}+y)a_{r-1}\\
\begin{split}
={}& d_{r-1}\big[ 1-f_{r-2}+f_{r-3}f_{r-2}-\cdots
+(-1)^{r-1}f_0\cdots f_{r-2}\big]\\
& +(-1)^rf_0\cdots f_{r-2}(z_{r-1}+y)a_{r-1},
\end{split}
\end{align*}
using a formula for $a_{r-1}=b_{r-1}$ from section \ref{poly}.
By definition of $a_0$, this gives that
\begin{multline*}
a_{r-1}=d_{r-1}\big[ 1-f_{r-2}+f_{r-3}f_{r-2}-\cdots
+(-1)^{r-2}f_1\cdots f_{r-2}\big]\\
+(-1)^{r-1}f_0\cdots f_{r-2}a_0.
\end{multline*}
Now successively apply the definitions of $a_1$, $a_2$, and so on,
giving
\begin{align*}
&\cdots\\
a_{r-1}&=d_{r-1}(1-f_{r-2}+f_{r-3}f_{r-2})-f_{r-4}f_{r-3}f_{r-2}a_{r-4}\\
&=d_{r-1}(1-f_{r-2})+f_{r-3}f_{r-2}a_{r-3}\\
&=d_{r-1}-f_{r-2}a_{r-2}.
\end{align*}
That is what we wanted.

\begin{lemma}
\label{equations}
The following $r$ equations hold.
\begin{align*}
a_{r-1}&=b_{r-1}\\
a_{r-2}&=e_{r-2}(1+a_{r-1})-a_{r-1}\\
a_{r-3}&=e_{r-3}(1+a_{r-1}+a_{r-2})-a_{r-2}\\
&\cdots\\
a_0&=e_0(1+a_{r-1}+\cdots+a_1)-a_1.
\end{align*}
\end{lemma}

The integer $y$ is at least 2. Given that, Lemma \ref{equations} shows
that the $a_i$'s are positive integers.
We will also use it to prove the identity
$d_{r-1}=yx$, which is important for Theorem \ref{kltgeneral}.

\begin{proof}
(Lemma \ref{equations})
The first equation,
$a_{r-1}=b_{r-1}$, holds by definition of $a_{r-1}$.
Next, if $r\geq 2$,
we want to show that $a_{r-2}=e_{r-2}(1+a_{r-1})-a_{r-1}$.
It suffices to prove this after multiplying by $f_{r-2}$,
which we do in order to use the result above
that $a_{r-1}=d_{r-1}-f_{r-2}a_{r-2}$.
So we want to show that 0 is equal to
\begin{align*}
{}&\; f_{r-2}e_{r-2}(1+a_{r-1})-f_{r-2}a_{r-1}-(d_{r-1}-a_{r-1})\\
{}=&\; a_{r-1}(f_{r-2}e_{r-2}-f_{r-2}+1)+(f_{r-2}e_{r-2}-d_{r-1}).\\
{}=&\; b_{r-1}(f_{r-1}-1)+(e_{r-1}-d_{r-1}),
\end{align*}
where we used the identities that $e_{r-1}=f_{r-1}+f_{r-2}-2$
and $e_{r-1}=f_{r-2}e_{r-2}$ from section \ref{poly}. By definition,
$d_{r-1}=e_{r-1}+b_{r-1}(f_{r-1}-1)$, and so
the desired equation holds. So we have
$a_{r-2}=e_{r-2}(1+a_{r-1})-a_{r-1}$.

Now suppose we have proved the equation in Lemma \ref{equations}
for $a_{i+1}$,
with $0\leq i\leq r-3$; let us prove it for $a_i$.
That is, we want to show that $a_i=e_i(1+a_{r-1}+\cdots+a_{i+1})-a_{i+1}$.
By definition, $a_{i+1}=d_{r-1}-f_ia_i$, and so $f_ia_i=d_{r-1}-a_{i+1}$.
It suffices to prove
the desired identity after multiplying by $f_i$; so we want to show
that 0 is equal to
\begin{align*}
&\; f_ie_i(1+a_{r-1}+\cdots+a_{i+1})-f_ia_{i+1}-(d_{r-1}-a_{i+1})\\
=&\; a_{i+1}(f_ie_i-f_i+1)+f_ie_i(1+a_{r-1}+\cdots+a_{i+2})-d_{r-1}\\
=&\; a_{i+1}(f_{i+1}-1)+e_{i+1}(1+a_{r-1}+\cdots+a_{i+2})-d_{r-1},
\end{align*}
using the identities $e_{i+1}=f_{i+1}+f_i-2$ and $e_{i+1}=f_ie_i$
from section \ref{poly}.

By induction, we know that $a_{i+1}=e_{i+1}(1+a_{r-1}+\cdots+a_{i+2})-a_{i+2}$.
So we want to show that 0 is equal to
$$a_{i+1}f_{i+1}+a_{i+2}-d_{r-1},$$
which is true by definition of $a_{i+2}$ (or, in the case $i=r-3$,
by the equality $a_{r-1}=d_{r-1}-a_{r-2}f_{r-2}$ which we proved).
That completes the proof of Lemma \ref{equations}.
\end{proof}

We return to the proof of Theorem \ref{kltgeneral}.
From Lemma \ref{equations}, it is clear
that $a_0,\ldots,a_{r-1}$ are positive integers.
Writing $x=1+a_0+\cdots+a_{r-1}$, let us show that $d_{r-1}=yx$,
part of the statement of the theorem.
If $r=1$, then $a_0=b_0=1$ and $d_{r-1}=2y$, so $x=1+a_0=2$
and we see that $d_{r-1}=yx$. If $r\geq 2$,
then Lemma \ref{equations} says that $a_0=e_0(1+a_{r-1}+\cdots+a_1)-a_1$,
where $e_0=y$. Equivalently, $x=1+a_0+\cdots+a_r$ satisfies
$a_1=yx-(y+1)a_0$. Here $a_1=d_{r-1}-(y+1)a_0$, by definition if $r\geq 3$
and by the formula $a_{r-1}=d_{r-1}-f_{r-2}a_{r-2}$ shown above
if $r=2$.
We conclude that $d_{r-1}=yx$.
Since $y=c_{n-r+2}-1=c_0\cdots c_{n-r+1}$ by the properties
of Sylvester's sequence, we can also say
that $d_{r-1}=c_0\cdots c_{n-r+1}x$.

As in the statement of the theorem, define
$a_{r+i}=c_0\cdots \widehat{c_i}\cdots c_{n-r+1}x$
for $0\leq i\leq n-r+1$.
We now show that the weighted projective space
$Y=P(a_0,\ldots,a_{n+1})$ is well-formed,
when $r$ is odd and at least 3.
That is, we have to show
that $\gcd(a_0,\ldots,\widehat{a_m},\ldots,a_{n+1})=1$
for each $0\leq m\leq n+1$.
It suffices to show that
$\gcd(a_0,\ldots,\widehat{a_m},\ldots,a_{r-1})=1$
for each $0\leq m\leq r-1$.

We first compute some of section \ref{poly}'s polynomial sequences
modulo $y$.
By induction on $i$, we have $f_i\equiv 1\pmod{y}$
for all $i\geq 0$. It follows that
$$b_i\equiv\begin{cases}
1\pmod{y}&\text{if }i\text{ is even}\\
0\pmod{y}&\text{if }i\text{ is odd.}
\end{cases}$$
Also by induction, we find that
$$f_i\equiv\begin{cases}
0\pmod{y+1}&\text{if }i\text{ is even}\\
2\pmod{y+1}&\text{if }i\text{ is odd,}
\end{cases}$$
and hence
$$b_i\equiv\begin{cases}
1\pmod{y+1}&\text{if }i=0\\
-1\pmod{y+1}&\text{if }i>0.
\end{cases}$$

From there, we can show that $\gcd(a_0,\ldots,a_{r-1})=1$,
a step towards our goal. Namely, if a prime number $p$
divides $a_j$ for all $0\leq j\leq r-1$,
then the formula for $a_0$ from Lemma \ref{equations}
shows that $p$ divides $e_0=y$.
But $a_{r-1}=b_{r-1}\equiv 1\pmod{y}$ (using that $r$ is odd),
contradicting that $p$ divides $a_{r-1}$. So we have shown
that $\gcd(a_0,\ldots,a_{r-1})=1$.

Using that, let us show that
$\gcd(a_0,\ldots,\widehat{a_m},\ldots,a_{r-1})=1$
for each $2\leq m\leq r-1$. (We handle the cases where
$m$ is 0 or 1 afterward.)
Let $p$ be a prime number that divides $a_j$
for all $0\leq j\leq r-1$ with $j\neq m$. The formula
for $a_0$ from Lemma \ref{equations} gives that $0\equiv y(1+a_m)\pmod{p}$.
As a result, the formula for $a_{m-1}$ from the lemma gives that
$0\equiv e_{m-1}(1+a_m)-a_m\equiv -a_m\pmod{p}$,
using that $e_{m-1}$ is a multiple of $y$. This contradicts
the fact that $\gcd(a_0,\ldots,a_{r-1})=1$.

Next, we show that $\gcd(a_0,a_2,\ldots,a_{r-1})=1$.
Let $p$ be a prime number that divides $a_0$ as well
as $a_j$ for all $2\leq j\leq r-1$.
The formula for $a_0$ from Lemma \ref{equations} gives that
$0\equiv y(1+a_1)-a_1\equiv y+(y-1)a_1
\pmod{p}$. The formula for $a_1$ from the lemma gives that
$a_1\equiv e_1=y(y+1)\pmod{p}$. Combining these,
we have $0\equiv y+(y-1)y(y+1)=y^3\pmod{p}$. So $p$ divides $y$.
But then $a_1\equiv 0\pmod{p}$, contradicting
that $\gcd(a_0,\ldots,a_{r-1})=1$.

Finally, we show that $\gcd(a_1,\ldots,a_{r-1})=1$.
Let $p$ be a prime number that divides $a_j$ for all
$1\leq j\leq r-1$. By the formula for $a_1$ from Lemma \ref{equations},
$p$ divides $e_1=y(y+1)$. If $p$ divides $y$ ($=e_0$), then the formula
for $a_0$ from the lemma gives that $p$ divides $a_0$,
contradicting that $\gcd(a_0,\ldots,a_{r-1})=1$.
So $p$ divides $y+1$. But $a_{r-1}=b_{r-1}\equiv -1\pmod{y+1}$
since $r$ is at least 3,
contradicting that $p$ divides $a_{r-1}$.
This completes the proof
that $\gcd(a_0,\ldots,\widehat{a_m},\ldots,a_{r-1})=1$
for each $0\leq m\leq r-1$. So $Y$ is well-formed.

Next, let us show that the general hypersurface $X$ of degree $d_{r-1}$
in $Y$ is quasi-smooth.
For each $i>r-1$, we know that $a_i$ divides $d_{r-1}$; also, $d_{r-1}$
is greater than each $a_i$. (For $a_0,\ldots,a_{r-1}$, that follows
from the fact that $d_{r-1}=y(1+a_0+\cdots+a_{r-1})$.)
Given this, Lemma \ref{cycle} shows that quasi-smoothness follows
from a cycle of $r$ congruences, namely that
$d_{r-1}-a_{r-1}\equiv 0\pmod{a_{r-2}}$,
$d_{r-1}-a_{r-2}\equiv 0\pmod{a_{r-3}}$,
\ldots, $d_{r-1}-a_1\equiv 0\pmod{a_0}$,
and $d_{r-1}-a_0\equiv 0\pmod{a_{r-1}}$.
These are immediate from the definitions of $a_i$, together with
the identity $a_{r-1}=d_{r-1}-f_{r-2}a_{r-2}$ which we proved.
So $X$ is quasi-smooth.
In particular, $X$ has only cyclic quotient singularities,
and so $X$ is klt.

Therefore, $K_X=O_X(d_{r-1}-\sum a_i)$. Here
\begin{align*}d_{r-1}-\sum_{i=r}^{n+1}a_i
&=c_0\cdots c_{n-r+1}x-\sum_{i=r}^{n+1}a_i\\
&=c_0\cdots c_{n-r+1}\bigg( 1-\sum_{i=0}^{n-r+1}1/c_i\bigg) x\\
&=x\\
&=1+a_0+\cdots+a_{r-1},
\end{align*}
and so $K_X=O_X(1)$. As a result,
\begin{align*}
\vol(K_X)&=\frac{d_{r-1}}{a_0\cdots a_{n+1}}\\
&=\frac{(c_0\cdots c_{n-r+1})x}
{(c_0\cdots c_{n-r+1})^{n-r+1}x^{n-r+2}
a_0\cdots a_{r-1}}\\
&=\frac{1}{y^{n-r}x^{n-r+1}a_0\cdots a_{r-1}}.
\end{align*}

In terms of $y=c_{n-r+2}-1$, we have $a_{r-1}>y^{2^{r-1}-1}$.
Use the $r$ equations from Lemma \ref{equations}
to estimate the other $a_i$'s.
By descending induction on $i$,
using that $e_j\geq y^{2^j}$ for each $j$, it follows that 
$a_i>y^{2^r-2^i-1}$ for $0\leq i\leq r-1$.
Therefore, $x\geq a_0>y^{2^r-2}$.
It follows that
$\vol(K_X)<1/y^{(2^r-1)n-1}=1/(c_{n-r+2}-1)^{(2^r-1)n-1}$.

There is a constant $c\doteq 1.264$ such that $c_i$ is the closest
integer to $c^{2^{i+1}}$ for all $i\geq 0$ \cite[equations
2.87 and 2.89]{GK}.
This implies the crude statement that $\vol(K_X)< 1/2^{2^n}$
for all $n\geq r-1$.
\end{proof}

\section{Klt Fano varieties with $H^0(X,-mK_X)=0$
for a large range of positive integers $m$}

We now construct klt Fano varieties such that $H^0(X,-mK_X)=0$
for a large range of positive integers $m$ (Theorem \ref{kltfano}).
This is of interest
in connection with Birkar's theorem on the boundedness of complements.
Namely, for each positive integer $n$, there is a positive integer $e=e_n$
such that for every klt Fano variety $X$ of dimension $n$,
the linear system $|-eK_X|$ is not empty, and in fact it contains
a divisor $M$ with mild singularities in the sense that the pair
$(X,\frac{1}{e}M)$
is log canonical \cite[Theorem 1.1]{Birkarcomp}.
Our examples show that $e_n$ must grow
at least doubly exponentially, roughly like $2^{2^n}$.

In low dimensions, our examples are good but not optimal.
In dimension 2, Theorem \ref{kltfano} gives
the klt Fano surface of degree 256 in $P^3(128,69,49,11)$.
The optimal bottom weight (for quasi-smooth hypersurfaces
of dimension 2 with $K_X=O_X(-1)$) is 13, which occurs in the examples
$X_{256}\subset P^3(128,81,35,13)$
and $X_{127}\subset P^3(57,35,23,13)$. In dimension 3,
Theorem \ref{kltfano} gives the klt Fano 3-fold
$$X_{336960}\subset P^4(168480,112320,46837,9101,223),$$
which has $K_X=O_X(-1)$.
The optimal bottom weight here is 407,
from Johnson and Koll\'ar's klt Fano 3-fold \cite[Remark 3]{JK}:
$$X_{37584}\subset P^4(18792,12528,5311,547,407).$$
So Theorem \ref{kltfano} has excellent asymptotics
in high dimensions, but it is not optimal.

Our klt Fano varieties also have fairly small volume of $-K_X$; but that
has no particular significance, because
the volume of klt Fano varieties in a given dimension
can be arbitrarily small. (For example, for any positive integer $a$,
the weighted projective plane $Y=P(2a+1,2a,2a-1)$
is a klt Fano surface with $\vol(-K_Y)=18a/(4a^2-1)$.)

The definition of our klt Fano varieties is much like that of
the klt varieties of general type in Theorem \ref{kltgeneral}.
Again,
let $c_0,c_1,\ldots$ be Sylvester's sequence;
see section \ref{kltsectr=3} for the properties of that sequence.
We use the five sequences of polynomials $f_i$, $e_i$, $b_i$,
$z_i$ and $d_i$ in $\Z[y]$ from section \ref{poly}.
The one slightly different polynomial we need here is $\d_i:=-e_i+b_i(f_i-1)$,
in place of $d_i=e_i+b_i(f_i-1)$.

\begin{theorem}\label{kltfano}
Let $r$ be an odd integer at least 3 and let $n$ be an integer at least $r-1$.
Define integers $a_0,\ldots,a_{n+1}$ as follows.
Let $y=c_{n-r+2}-1$ and
\begin{align*}
a_{r-1}&=b_{r-1}\\
a_0&=\d_{r-1}-(z_{r-1}-y)a_{r-1}\\
a_1&=\d_{r-1}-f_0a_0\\
&\cdots\\
a_{r-2}&=\d_{r-1}-f_{r-3}a_{r-3}\\
\end{align*}
These are positive integers. Let
$x=-1+a_0+\cdots+a_{r-1}$; then
$\d_{r-1}=yx=c_0\cdots c_{n-r+1}x$.
Let $a_{r+i}=c_0\cdots \widehat{c_i}\cdots c_{n-r+1}x$
for $0\leq i\leq n-r+1$. Let $X$ be a general hypersurface
of degree $\d_{r-1}$
in the complex weighted projective space $P(a_0,\ldots,a_{n+1})$.
Then $X$ is a klt Fano variety of dimension $n$, and
$$H^0(X,-mK_X)=0$$
for all $1\leq m < b_{r-1}$.
Here $b_{r-1}\geq (c_{n-r+2}-1)^{2^{r-1}-1}$.
\end{theorem}

Taking $r=n+1$ if $n$ is even and $r=n$ if $n$ is odd, we deduce that
the bottom weight $b_{r-1}$ is at least $2^{2^n-1}$ if $n$ is even
and at least $6^{2^{n-1}-1}$ if $n$ is odd.

\begin{proof}
The proof is identical to that of Theorem \ref{kltgeneral},
with sign changes where needed. For example, in place of the identity
$d_i=(-1)^i+f_0\cdots f_{i-1}(z_i+y)$, use that
$\d_i=(-1)^i+f_0\cdots f_{i-1}(z_i-y)$. As in the proof
of Theorem \ref{kltgeneral}, start by showing
that $a_{r-1}=\d_{r-1}-f_{r-2}a_{r-2}$. The analog of Lemma \ref{equations}
says that
\begin{align*}
a_{r-1}&=b_{r-1}\\
a_{r-2}&=e_{r-2}(-1+a_{r-1})-a_{r-1}\\
a_{r-3}&=e_{r-3}(-1+a_{r-1}+a_{r-2})-a_{r-2}\\
&\cdots\\
a_0&=e_0(-1+a_{r-1}+\cdots+a_1)-a_1.
\end{align*}
That makes it clear that the $a_i$'s are positive integers.
The rest of the proof shows
that $X$ is a well-formed quasi-smooth hypersurface with $K_X=O_X(-1)$,
and its bottom weight is $b_{r-1}$.
\end{proof}

% Omit these bibliography lines if there's no bibliography.

\small \sc UCLA Mathematics Department, Box 951555,
Los Angeles, CA 90095-1555

totaro@math.ucla.edu

chwang@math.ucla.edu
\end{document}